\documentclass{amsart}

\newif\ifdraft\draftfalse
\newif\ifcite\citefalse

\newif\ifblow\blowtrue

\usepackage{amsfonts,amssymb, amsmath}
\usepackage{amscd}
\usepackage{mathrsfs} 
\usepackage[hyperindex]{hyperref}  
\usepackage[alphabetic]{amsrefs}
\ifdraft\ifcite\usepackage{showkeys}\else\usepackage[notcite,notref]{showkeys}\fi\fi

\usepackage{xypic}
\newdir{ >}{{}*!/-5pt/\dir{>}}

\newtheorem{theorem}
{Theorem}

\theoremstyle{remark}

\theoremstyle{definition}

\theoremstyle{remark}

\newtheorem{remark}[equation]{Remark}

\numberwithin{equation}{section}

\def\ch{{\mathcal H}}

\def\br{\mathbb R}
\def\bc{\mathbb C}
\def\bp{\mathbb P}

\def\bean{\begin{eqnarray}}
\def\eean{\end{eqnarray}}
\def\bea{\begin{eqnarray*}}
\def\eea{\end{eqnarray*}}
\def\l{\label}

\def\x0{X_{s_0}}
\def\2p{\bp^1\times \bp^1}

\def\a{\alpha}

\def\o{\omega}
\def\nm{\nonumber}
\def\er{\eqref}

\def\ben{\begin{equation}}
\def\een{\end{equation}}

\def\fg{\mathfrak{g}}
\def\fh{\mathfrak{h}}

\def\fm{\mathfrak{m}}

\def\fo{\mathfrak{o}}

\def\fs{\mathfrak{s}}

\newcommand\ii{\mathcal I}

\newcommand\ban{\begin{proof}[Answer]}
\newcommand\ean{\end{proof}}

\newcommand\ad{\mathop{\mathrm{ad}}\nolimits}

\newcommand\om{\omega}

\begin{document}
\title{On transgression in associated bundles}
\author{Zhaohu Nie}
\email{znie@psu.edu}
\address{Department of Mathematics\\
Penn State Altoona\\
3000 Ivyside Park\\
Altoona, PA 16601, USA}

\date{\today}

\begin{abstract}
We formulate and prove a formula for transgressing characteristic forms in general associated bundles following a method of Chern~\cite{chern90}. As applications, we derive D. Johnson's explicit formula in \cite{johnson} for such general transgression and Chern's first transgression formula in \cite{chern2} for the Euler class.
\end{abstract}

\maketitle

\section{Introduction}

Let $G$ be a Lie group with Lie algebra $\fg$, $M$ a manifold, and $\pi:E\to M$ a principal $G$-bundle over $M$. A connection on $E$ is given by a $\fg$-valued 1-form $\omega$ on $E$ satisfying certain conditions. Its curvature form is a $\fg$-valued 2-form on $E$ defined by  
\begin{equation}
\label{structural}
\Omega=d\omega+\frac{1}{2}[\omega,\omega].
\end{equation}
For $P\in \ii(\fg)$ an $\ad_G$-invariant polynomial on $\fg$, the form $P(\Omega)$, 
\emph{a priori} defined on $E$, is horizontal and invariant and so naturally defines a form on $M$. $P(\Omega)$ is closed and its cohomology class is independent of the choice of the connection $\omega$. As such, it is called a characteristic form of $E$.

Chern-Simons \cite{cs} transgressed $P(\Omega)$ in the principal bundle $E$. That is, they showed that $P(\Omega)$ is a coboundary in $E$ by canonically constructing a form $TP(\omega)$ ($T$ for transgression), depending on the connection $\omega$,
such that $dTP(\omega)=P(\Omega)$. The Chern-Simons forms $TP(\omega)$ define important secondary invariants and appear naturally in questions involving manifolds with boundaries.

It is also important to be able to transgress the characteristic form $P(\Omega)$ in ``smaller" bundles, that is, associated bundles. Let $H<G$ be a subgroup with Lie algebra $\fh$. We assume that the homogenous space $G/H$ is \emph{reductive} (see \er{decomp}). Consider the associated bundle $B=E\times_G (G/H)$, which fits in the following commutative diagram 
$$
\xymatrix{
E\ar[rd]_\pi\ar[rr]^{\pi_1} & & B\ar[ld]^{\pi_2}\\
 & M.& 
 }$$
Here $\pi_1:E\to B$ is a principal bundle with structure group $H$. The question is then to find a canonical form $TP(\omega)$ on the associated bundle $B$, which transgresses $P(\Omega)$ (regarded as on $B$), at least in many important cases. (See the precise statement in \er{new relation}.) 

Chern \cite{chern90} solved this question for the important special case of Chern classes $c_s$ with the relevant groups being $GL(s-1;\bc)< GL(q;\bc)$ for $1\leq s\leq q$, using a deformation trick with its root in \cite{cs}.  On the other hand, in the general situation, D. Johnson \cite{johnson} explicitly constructed $TP(\omega)$ by (rather complicated) recurrence. 

In this paper, we first formulate and prove a transgression formula in general associated bundles in Section 2. The setup we consider follows that of D. Johnson \cite{johnson}, but the idea of construction comes from Chern~\cite{chern90}. Then we derive D. Johnson's explicit transgression formula in \cite{johnson} rather easily in Section 3. In Section 4, we show that the very first transgression formula which started the whole business, that is, Chern's formula  in \cite{chern2} for transgressing the Euler class in the unit tangent sphere bundle, can be obtained using our general method. This shows compatibility of these differential-geometric notions of transgression.

The authors thanks the referee for careful reading and helpful comments. 

\section{General transgression formula}

With notation as above and following D. Johnson \cite{johnson}, we choose and fix an $\ad_H$-invariant decomposition 
\ben\label{decomp}
\fg=\fh+\fm,
\een
since we assume that $G/H$ is reductive (see \cite[\S X.2]{KN2}). 
We denote the corresponding projections by $p_\fh:\fg\to \fh$ and $p_\fm:\fg\to \fm$. 
Then one has the following decomposition of $\omega$
\begin{gather}
\omega=\psi+\phi,\label{opp}\\
\psi=p_\fh\circ \omega\in \Lambda^1(E,\fh),\ \phi=p_\fm\circ \omega\in \Lambda^1(E,\fm).\label{comp proj}
\end{gather}
$\psi$ is easily seen to be a connection form on $\pi_1:E\to B$
with curvature form $\Psi=d\psi+\frac{1}{2}[\psi,\psi]$. 
By~\er{structural} and~\er{opp}, we calculate the curvature form of $\omega$ to be 
\begin{align}
\Omega&=d(\psi+\phi)+\frac{1}{2}[\psi+\phi,\psi+\phi]\nm\\
	&=d\psi+\frac{1}{2}[\psi,\psi]+d\phi+[\psi,\phi]+\frac{1}{2}[\phi,\phi]\nm\\
	&=\Psi+d_\ch \phi+\frac{1}{2}[\phi,\phi],\label{eq-O}
\end{align}
where
\ben\label{phi-cov}
d_\ch \phi=d\phi+[\psi,\phi]
\een
is the $\psi$-covariant derivative of $\phi$. One then has from \er{eq-O}
\begin{equation}
\label{write-back}
d_\ch \phi=\Omega-\Psi-\frac{1}{2}[\phi,\phi].
\end{equation}

Following \cite{chern90}, consider the following family of differential forms on $E$
\ben\label{ot}
\omega(t)=\psi+t\phi,\ 0\leq t\leq 1. 
\een
Following~\er{structural}, define
\ben\label{t}
\Omega(t)=d\omega(t)+\frac{1}{2}[\omega(t),\omega(t)].
\een
Similar to ~\er{eq-O}, one calculates 
\begin{align}
\Omega(t)&=\Psi+t\, d_\ch \phi+\frac{1}{2}t^2[\phi,\phi]\label{differential}\\
&=(1-t)\Psi-\frac{1}{2}t(1-t)[\phi,\phi]+t\Omega.\label{Omega(t)}
\end{align}
where the last equality uses~\er{write-back}. 

We polarize an $\ad_G$-invariant polynomial $P\in \ii^k(\fg)$ of degree $k$ to a symmetric, multi-linear function $P:\underbrace{\fg\otimes\cdots\otimes\fg}_{k}\to \br$. 
Chern \cite{chern90} proved the following theorem in the special case of Chern classes, where the $P(\Psi)$ term in \er{new relation} is 0.
\begin{theorem}
\label{general formula} 
There is a canonical differential form on the associated bundle $B$ defined by 
\ben\label{tpo}
TP(\omega)=k\int_0^1 P(\phi,\Omega(t),\dots,\Omega(t))\, dt
\een
such that 
\ben
\label{new relation}
dTP(\omega)=P(\Omega)-P(\Psi).
\een
\end{theorem}


\begin{proof} We first show that $TP(\omega)$ in \er{tpo} defines naturally a differential form on $B$ by showing that it is invariant and horizontal for the principal bundle $\pi_1:E\to B$ with structure group $H$ and connection $\psi$. It is easy to see that for the right multiplication $R_h$ by $h\in H$, one has
$$
R_h^*\psi=\ad_{h^{-1}}\psi,\ R_h^*\phi=\ad_{h^{-1}}\phi,
$$
since $\omega$ is a connection form and the decomposition in \er{decomp} is $\ad_H$-invariant. Therefore by \er{ot} and \er{t}, one has $R_h^*\Omega(t)=\ad_{h^{-1}}\Omega(t)$. The invariance of $P$ under $\ad_G$ and hence under $\ad_H$ then shows that $TP(\omega)$ is invariant. It is also easy to see that $\phi$ is horizontal from its definition in \er{comp proj}. The horizontality of curvature forms $\Psi$ and $\Omega$ then implies the horizontality of $\Omega(t)$ from \er{Omega(t)} and that of $TP(\omega)$.  

In~\er{Omega(t)}, $\Omega(1)=\Omega$ and $\Omega(0)=\Psi$. Therefore to prove \er{new relation}, we only need to show that 
\ben\label{to prove}
\frac{\partial}{\partial t}P(\Omega(t))=k\,dP(\phi,\Omega(t),\dots,\Omega(t)).
\een

One computes
\begin{align*}
 & \frac{\partial}{\partial t}P(\Omega(t))=kP\left(\frac{\partial}{\partial t}\Omega(t),\Omega(t),\dots,\Omega(t)\right)\\
 =&kP(d_\ch \phi+t[\phi,\phi],\Omega(t),\dots,\Omega(t))
 \end{align*}
 by \er{differential}. Also
 \begin{align*}
 &dP(\phi,\Omega(t),\dots,\Omega(t))\\
 =&P(d_\ch\phi,\Omega(t),\dots,\Omega(t))-(k-1)P(\phi,d_\ch\Omega(t),\Omega(t),\dots,\Omega(t))\\
 =&P(d_\ch\phi,\Omega(t),\dots,\Omega(t))-(k-1)P(\phi,t[\Omega(t),\phi],\Omega(t),\dots,\Omega(t))\\
=&P(d_\ch\phi,\Omega(t),\dots,\Omega(t))+P(t[\phi,\phi],\Omega(t),\dots,\Omega(t))\\
=&P(d_\ch\phi+t[\phi,\phi],\Omega(t),\dots,\Omega(t)),
\end{align*}
where the second equality uses 
\begin{align*}
&d_\ch \Omega(t)=d\Omega(t)+[\psi,\Omega(t)]\\
=&[d\omega(t),\omega(t)]-[\Omega(t),\psi]=[\Omega(t),\omega(t)]-[\Omega(t),\psi]\\
=&t[\Omega(t),\phi]
\end{align*}
by differentiating \er{t} and using \er{ot}, and the third uses the $\ad_G$-invariance of $P$, which implies
$$
P([\phi,\phi],\Omega(t),\dots,\Omega(t))+(k-1)P(\phi,[\Omega(t),\phi],\Omega(t),\dots,\Omega(t))=0.
$$
\end{proof}

\begin{remark} In the spirit of the modern approach to transgression (see, e.g., \cite[Def. 1.8]{BM}), we can interpret Theorem \ref{general formula} as follows. Consider the following fibration $\pi_1\times Id_\br: E\times \br\to B\times \br$.  
Let $p_E: E\times \br\to E$ be the natural projection. 
Consider the differential form $\tilde \omega\in \Lambda^1(E\times \br,\fg)$ such that 
$$\tilde \omega\left(\frac{\partial}{\partial t}\right)=0,\ \tilde \omega|_{E\times\{t\}}=p_E^*\omega(t)$$ 
for $t\in \br$ with $\omega(t)$ defined in \er{ot}.  
Then following \er{structural}, define 
$$\tilde\Omega=d^{E\times \br}\tilde\om+\frac{1}{2}[\tilde\om,\tilde\om].$$ 
In comparison with \er{t} and using \er{ot}, we see that
\ben\l{omt}
\tilde\Omega=dt\wedge p_E^*\phi+p_E^*\Omega(t).
\een
For the invariant polynomial $P\in {\mathcal I}^k(\fg)$, $P(\tilde\Omega)$ is a well-defined form on $B\times \br$, and 
\ben\l{d=0}
d^{B\times \br}P(\tilde\Omega)=0
\een
by the standard Chern-Weil theory. By \er{omt}, separate the terms in $P(\tilde \Omega)$ without or with $dt$ and we have explicitly
\ben\label{inner}
P(\tilde \Omega)=p_E^* P(\Omega(t))+dt\wedge \bigl(k\, p_E^* P(\phi,\Omega(t),\cdots,\Omega(t))\bigr).
\een 
Computing the coefficients of $dt$ in \er{d=0} and using \er{inner}, we immediately see that 
$$
\frac{\partial}{\partial t}P(\Omega(t))=k\, d^BP(\phi,\Omega(t),\cdots,\Omega(t)).
$$
This reproves \er{to prove} and hence Theorem \ref{general formula}.  
\end{remark}

\section{D. Johnson's explicit formula}

As in~\cite[(3.5)]{cs}, one can evaluate the integral formula in \er{tpo} using \er{Omega(t)}, which then gives the following explicit formula of D. Johnson \cite{johnson}, who proved it using recurrence. 
\begin{theorem}[D. Johnson \cite{johnson}]\label{expansion}
The explicit formula for the transgression form on an associated bundle is
\ben\label{dj's}
TP(\omega)=\sum_{i=0}^{k-1}\sum_{j=0}^{k-i-1} A_{ij} P(\phi,[\phi,\phi]^i,\Psi^j,\Omega^{k-i-j-1}), 
\een
where $A_{ij}=(-1)^i \frac{k!(k-j-1)!(i+j)!}{2^i i!j!(k-i-j-1)!(k+i)!}$. 
\end{theorem}
\begin{proof}[Proof using Theorem \ref{general formula}] Plugging \er{Omega(t)} into \er{tpo}, applying the multinomial theorem for the $(k-1)$ arguments of $\Omega(t)$ (since $P$ is symmetric and the differential forms involved are of degree two), and applying some basic knowledge about the beta functions, one gets
\begin{align*}
&TP(\omega)=k\int_0^1 P\biggl(\phi,\underbrace{-\frac{1}{2}t(1-t)[\phi,\phi]+(1-t)\Psi+t\Omega,\dots}_{k-1}\biggr)\,dt\\
=&k\sum_{i=0}^{k-1}\sum_{j=0}^{k-i-1}\frac{(k-1)!}{i!j!(k-i-j-1)!}\int_0^1 \Bigl(-\frac{1}{2}t(1-t)\Bigr)^i (1-t)^j t^{k-i-j-1}\, dt\\
& \qquad\qquad\quad P(\phi,[\phi,\phi]^i,\Psi^j,\Omega^{k-i-j-1})\\
=&k\sum_{i=0}^{k-1}\sum_{j=0}^{k-i-1}\frac{(k-1)!}{i!j!(k-i-j-1)!}\Bigl(-\frac{1}{2}\Bigr)^i\int_0^1 t^{k-j-1} (1-t)^{i+j}\, dt\\
&\qquad\qquad\quad P(\phi,[\phi,\phi]^i,\Psi^j,\Omega^{k-i-j-1})\\
=&\sum_{i=0}^{k-1}\sum_{j=0}^{k-i-1}  \frac{k!}{i!j!(k-i-j-1)!} \Bigl(-\frac{1}{2}\Bigr)^i\frac{(k-j-1)!(i+j)!}{(k+i)!} \\
&\ \ \ \ \quad\quad\quad P(\phi,[\phi,\phi]^i,\Psi^j,\Omega^{k-i-j-1}).
\end{align*}
This is \er{dj's}. 
\end{proof}

\section{Chern's first transgression formula}

As another application, we work out Chern's first transgression formula in \cite{chern2} for the Euler form using Theorem \ref{general formula}. (\cite{chern2} reformulates and simplifies the truly ``first'' formula in \cite{chern}.) 

Assume $M$ has even dimension $n=2k$. For the Euler form, the relevant Lie groups are $\mathrm{SO}(n-1)< \mathrm{SO}(n)$, and the associated bundle $B\to M$ is the unit tangent sphere bundle $STM\to M$. We use the convention that $i$ ranges from 1 to $n$, and $\alpha, \beta$ range from $1$ to $n-1$. The invariant polynomial of degree $k$ on $\fs\fo(n)$ for the Euler form is, up to the scale $\frac{1}{(2\pi)^k}$ to make the class integral, the Pfaffian defined  by 
\ben
\label{which P?}
Pf(A)=\frac{1}{2^k k!}\sum_{i}\epsilon(i)A_{i_1i_2}\dots A_{i_{n-1}i_n}
\een
for $A\in \fs\fo(n)$, where the summation ranges over permutations $i$ of $\{1,\dots,n\}$ and $\epsilon(i)$ is the sign of $i$. 
(In this paper, products of  differential forms always mean ``exterior products" although we omit the notation $\wedge$ for simplicity.) 

\begin{theorem}[Chern \cite{chern2}] 
There is a differential form 
on the unit tangent sphere bundle $STM$ defined by 
\ben 
\label{chern's}
TPf(\omega)=\sum_{j=0}^{k-1}  \frac{1}{2^{j}j!(2k-2j-1)!!} 
\sum_\a \epsilon(\a) \Omega_{\a_1 \a_2}\dots \Omega_{\a_{2j-1}\a_{2j}}\o_{\a_{2j+1}n}\dots\o_{\a_{n-1}n}
\een 
such that 
\ben
\label{no Psi}
dTPf(\omega)=Pf(\Omega).
\een
\end{theorem}

\begin{remark} Unfortunately there some confusing sign discrepancies in the literature. To the author, these are caused by Chern's nonstandard choice (from the viewpoint of principal bundles) of subindices to indicate rows and columns. His connection matrices $(\omega_{ij})$ and curvature matrices $(\Omega_{ij})$ (see for example \cite[(2)]{chern2}) are the \emph{transposes} of the standard ones. Since these matrices are skew-symmetric, a lot of signs are generated this way. Our formula \er{chern's} is actually the modified one by H. Wu \cite[(10)]{Wu} in his survey of Chern's work. 
\end{remark}


\begin{proof}[Proof using Theorem \ref{general formula}]  Write the connection form $\omega$ with values in skew-symmetric matrices as
$$
\omega=\left(\begin{matrix}
\omega_{\a\beta} & \omega_{\alpha n}\\
-\omega_{\beta n} & 0
\end{matrix}\right).
$$
By \er{opp} and with the obvious decomposition, we have
\ben
\label{opp-wo} 
\psi=\left(\begin{matrix}
\omega_{\a\beta} & 0\\
0 & 0
\end{matrix}\right),\ 
\phi=\left(\begin{matrix}
0 & \omega_{\alpha n}\\
-\omega_{\beta n} & 0
\end{matrix}\right).
\een
We compute 
\ben
\label{to plug in}
-\frac 1 2 [\phi,\phi]=\left(\begin{matrix}
\omega_{\a n}\o_{\beta n} & 0\\
0 & 0
\end{matrix}\right),\ 
\Omega=\left(\begin{matrix}
\Omega_{\a\beta} & \Omega_{\a n}\\
\Omega_{n\beta} & 0
\end{matrix}\right).
\een
When applying formula \er{tpo} for $Pf$ in \er{which P?}, one has the factor $\phi_{i_1i_2}$. Therefore  in view of \er{opp-wo}, one of the $i_1,i_2$ must be $n$ for $\phi_{i_1i_2}$ to be nonzero.  As a result, one only cares about the $(\a,\beta)$-elements of $\Omega(t)$ in \er{tpo} in our situation.  
Note that from \er{phi-cov} and \er{opp-wo}, the $\psi$-covariant derivative of $\phi$
\begin{align*}
&d_\ch\phi=d\phi+[\psi,\phi]
=\left(\begin{matrix}
0 & \Omega_{\a n}\\
\Omega_{n\beta} & 0
\end{matrix}
\right)
\end{align*}
has trivial $(\a,\beta)$-elements. 
Hence we say $d_\ch\phi\equiv 0$ with $``\equiv"$ meaning ``having the same $(\a,\beta)$-elements". Then \er{eq-O} gives $\Psi\equiv \Omega-\frac 1 2 [\phi,\phi]$ and \er{differential} gives
\ben
\label{simpler ot}
\Omega(t)\equiv \Omega-\frac 1 2(1-t^2)[\phi,\phi].
\een

We now compute formula \er{tpo} using \er{simpler ot} and the idea for the proof of Theorem \ref{expansion}, rather than computing directly \er{dj's} which may be harder. That is, we will apply the binomial theorem and some basic integral formula, which in this case is
$$
\int_0^1 \left(1-t^2\right)^{k-j-1}\, dt= \frac{(2k-2j-2)!!}{(2k-2j-1)!!}
$$
by for example trigonometric substitution and induction using integration by parts. 
We proceed as follows:
{\allowdisplaybreaks
\begin{align*}
&TPf(\omega)=k\int_0^1 Pf\biggl(\phi,\underbrace{\Omega-\frac 1 2(1-t^2)[\phi,\phi],\dots}_{k-1}\biggr)\,dt\\
=&k\sum_{j=0}^{k-1} \frac{(k-1)!}{j!(k-j-1)!}\int_0^1 \bigl(1-t^2\bigr)^{k-j-1}\, dt\, Pf\biggl(\phi,\Omega^j,\Bigl(-\frac 1 2 [\phi,\phi]\Bigr)^{k-j-1}\biggr)\\
=&\sum_{j=0}^{k-1}  \frac{k!}{j!(k-j-1)!} \frac{(2k-2j-2)!!}{(2k-2j-1)!!}\,Pf\biggl(\phi,\Omega^j,\Bigl(-\frac 1 2 [\phi,\phi]\Bigr)^{k-j-1}\biggr)\\
=&\sum_{j=0}^{k-1}  \frac{k!}{j!(k-j-1)!} \frac{2^{k-j-1}(k-j-1)!}{(2k-2j-1)!!} \frac {1} {2^k k!}\\
& \qquad \sum_i \epsilon(i) \phi_{i_1 i_2}\Omega_{i_3 i_4}\dots \Omega_{i_{2j+1}i_{2j+2}}\o_{i_{2j+3}n}\o_{i_{2j+4}n}\dots\o_{i_{n-1}n}\o_{i_n n}\\
=&\sum_{j=0}^{k-1}  \frac{ 1}{2^{j+1}j!(2k-2j-1)!!} {\bf 2}\sum_\a \epsilon(\a) \omega_{\a_1 n}\Omega_{\a_2 \a_3}\dots \Omega_{\a_{2j}\a_{2j+1}}\o_{\a_{2j+2}n}\dots\o_{\a_{n-1}n},
\end{align*}
}
where the second last equality follows from \er{which P?} and \er{to plug in}, and the last equality follows after some cancellation of coefficients and from \er{opp-wo}. Here the {\bf 2} appears because $n$ can appear as $i_1$ or $i_2$ in $\phi_{i_1i_2}$ and the summation ranges over permutations $\alpha$ of $\{1,\dots,n-1\}$. 
The last expression is exactly \er{chern's}, after canceling the {\bf 2} and moving the $\omega_{\a_1 n}$ after the $\Omega$'s. 

Now $\Psi=d\psi+\frac 1 2 [\psi,\psi]$ clearly has the last row and column zero from \er{opp-wo}. One then has $Pf(\Psi)=0$ for our $Pf$ in \er{which P?}. Therefore \er{new relation}
implies \er{no Psi}.
\end{proof}

\begin{remark} The compatibility of Chern's transgression form in \cite{chern2} with the modern transgressed Euler form, the Mathai-Quillen form, is explained for example in \cite[\S 3.5]{Z}. 
\end{remark}

\bigskip


\begin{bibdiv}
\begin{biblist}

\bib{BM}{article}{
   author={Br{\"u}ning, J.},
   author={Ma, Xiaonan},
   title={An anomaly formula for Ray-Singer metrics on manifolds with
   boundary},
   journal={Geom. Funct. Anal.},
   volume={16},
   date={2006},
   number={4},
   pages={767--837},
   issn={1016-443X},
}

\bib{chern}
{article}{
   author={Chern, Shiing-shen},
   title={A simple intrinsic proof of the Gauss-Bonnet formula for closed
   Riemannian manifolds},
   journal={Ann. of Math. (2)},
   volume={45},
   date={1944},
   pages={747--752},
   issn={0003-486X},
}

\bib{chern2}{article}{
   author={Chern, Shiing-shen},
   title={On the curvatura integra in a Riemannian manifold},
   journal={Ann. of Math. (2)},
   volume={46},
   date={1945},
   pages={674--684},
   issn={0003-486X},
}

\bib{chern90}
{article}{
   author={Chern, Shiing-shen},
   title={Transgression in associated bundles},
   journal={Internat. J. Math.},
   volume={2},
   date={1991},
   number={4},
   pages={383--393},
   issn={0129-167X},
}

\bib{cs}
{article}{
   author={Chern, Shiing-shen},
   author={Simons, James},
   title={Characteristic forms and geometric invariants},
   journal={Ann. of Math. (2)},
   volume={99},
   date={1974},
   pages={48--69},
   issn={0003-486X},
}


\bib{johnson}
{article}{
   author={Johnson, David L.},
   title={Chern-Simons forms on associated bundles, and boundary terms},
   journal={Geom. Dedicata},
   volume={128},
   date={2007},
   pages={39--54},
   issn={0046-5755},
}

\bib{KN2}{book}{
   author={Kobayashi, Shoshichi},
   author={Nomizu, Katsumi},
   title={Foundations of differential geometry. Vol. II},
   series={Interscience Tracts in Pure and Applied Mathematics, No. 15 Vol.
   II },
   publisher={Interscience Publishers John Wiley \& Sons, Inc., New
   York-London-Sydney},
   date={1969},
   pages={xv+470},
}

\bib{Wu}{article}{
   author={Wu, H.},
   title={Shiing-shen Chern: 1911--2004},
   journal={Bull. Amer. Math. Soc. (N.S.)},
   volume={46},
   date={2009},
   number={2},
   pages={327--338},
   issn={0273-0979},
}


\bib{Z}{book}{
   author={Zhang, Weiping},
   title={Lectures on Chern-Weil theory and Witten deformations},
   series={Nankai Tracts in Mathematics},
   volume={4},
   publisher={World Scientific Publishing Co. Inc.},
   place={River Edge, NJ},
   date={2001},
   pages={xii+117},
   isbn={981-02-4686-2},
}

		
\end{biblist}
\end{bibdiv}
\end{document}